\documentclass{article}

\usepackage{arxiv}

\usepackage{amssymb}
\usepackage{amsmath}

\usepackage{pst-node}
\usepackage{subfigure}
\usepackage{graphicx}
\usepackage{times}
\usepackage{amsmath}\usepackage{appendix}
\usepackage{lipsum}

\usepackage{amsthm}
\theoremstyle{plain}
\newtheorem{theorem}{Theorem}[section]
\newtheorem{lemma}[theorem]{Lemma}

\newtheorem{proposition}[theorem]{Proposition}
\newtheorem{corollary}[theorem]{Corollary}

\theoremstyle{definition}
\newtheorem{definition}[theorem]{Definition}
\newtheorem{example}[theorem]{Example}

\newcommand{\R}{{\mathbb R}}

\newcommand{\IR}{{\bf I}{\mathbb R}}

\newcommand{\sle}{\sqsubseteq}
\newcommand{\M}{{\bf M}}

\newcommand{\dda}{\mathord{\hbox{\makebox[0pt][l]{\lower .6mm
                           \hbox{$\downarrow$}}$\downarrow$}}}
\newcommand{\dua}{\mathord{\hbox{\makebox[0pt][l]{\raise .6mm
                           \hbox{$\uparrow$}}$\uparrow$}}}

\begin{document}

\setlength{\pdfpageheight}{\paperheight}
\setlength{\pdfpagewidth}{\paperwidth}




\title{A generalisation of Henstock-Kurzweil integral to compact metric spaces}

\author{Abbas Edalat\\
           Department of Computing\\ Imperial College London\\
           {ae@ic.ac.uk}}

\maketitle
\date{}
\begin{abstract}
   
We introduce the notion of a gauge and of
a tagged partition (subordinate to a given gauge) by intersections of
open and closed sets of a compact metric space extending the
corresponding notions in Henstock-Kurzweil integration of real-valued
functions with respect to the Lebesgue measure on the unit interval. We
show that, for the integration of bounded functions with respect to a
normalised Borel measure $\mu$ on a compact metric space, the notion of a
gauge and an associated tagged partition, arise naturally from a
normalised simple valuation way-below the Borel measure. Then we consider the integration of unbounded functions with respect to a
normalised Borel measure on a compact metric space, for which the
Lebesgue integral may fail to exist. A pair of a tagged
partition and a gauge defines a simple valuation and we introduce a
partial order on these pairs, emulating the partial order of simple
valuations in the probabilistic power domain. We define the $D_\mu$-integral
of a real-valued function with respect to a Borel measure using the
limit of the net of the integrals of the simple valuations induced by
pairs of tagged partitions and gauges for the function. The $D_\mu$-integral
of functions on a compact metric space with respect to a normalised
Borel measure satisfies the basic properties of an integral and
generalises the Henstock-Kurzweil integral. We show that when the
Lebesgue integral of the function exists then the $D_\mu$-integral also exists
and they have the same value. We provide a family of real-valued
functions on the Cantor space that are $D_\mu$-integrable but not Lebesgue
integrable.
\end{abstract}

Key words:
Partition-Gauge Pairs;
$D_\mu$-Integral;
Lebesgue Integral;
Henstock-Kurzweil Integral

\section{Introduction}

In this the paper, we develop a generalisation of Henstock-Kurzweil integral for real-valued functions on compact metric spaces with respect to any bounded Borel measure. We first recall a few historical notes in this matter. In 1902, Henri Lebesgue introduced in his PhD thesis what is now called Lebesgue integration. It was clear from the outset that a simple function such as,
\[x\mapsto \frac{1}{x}\sin\frac{1}{x^3}:[0,1]\to \R_\bot,\] which is intuitively integrable since it has an improper Riemann integral as $x\to 0^+$, is not Lebesgue integrable as it is not absolutely integrable. 

To redress this situation, Arnaud Denjoy in 1912 and Oskar Perron in 1914 introduced a generalisation of the Lebesgue theory in which unbounded functions can have integrals without being absolutely integrable. While these two theories turned out to be equivalent, they were too complicated to enter into mainstream mathematics. In the 1950's Ralph Henstock and Jaroslav Kurzweil, working independently, developed a much simpler integration theory, now called Henstock-Kurzweil integration, which is equivalent to the integration theory of Denjoy and Perron. We now recall the definition of the Henstock-Kurzweil integral of a function~\cite{Bar01}.

A {\em gauge} on $[0,1]$ is a function $\gamma:[0,1]\to \R^+$, which is meant to generalise the norm of a partition of $[0,1]$ to one that depends on each point $x\in [0,1]$. A {\em tagged partition} $P$ of $[0,1]$ is a finite collection $(t_i,I_i)_{i=1}^K$ where $(I_i)_{i=1}^K$ is a partition of $[0,1]$ by the closed intervals $I_i$ for $1\leq i\leq K$ and $t_k\in I_k$ for each  $1\leq i\leq K$. The tagged partition $P=(t_i,I_i)_{i=1}^K$ is said to be {\em $\gamma$-fine} or {\em subordinate to $\gamma$} if $I_i\subset (t_i-\gamma(t_i),t_i+\gamma(t_i))$ for all $1\leq i\leq K$. Given a function $f:[0,1]\to \R_\bot$, the Riemann sum of the tagged partition $P$ for $f$ is given by,\[S(f,P)=\sum_{i=1}^Kf(t_k)\ell(I_k),\] where $\ell(I)$ is the length of the interval $I$. A function $f:[0,1]\to \R_\bot$ has Henstock-Kurzweil integral $a\in\R$ if for each $\epsilon>0$, there exists a gauge $\gamma$ such that $|a-S(f,P)|<\epsilon$ for all $\gamma$-fine tagged partitions of $[0,1]$. The Henstock-Kurzweil integral, which can also be defined for real-valued function on $\R^n$, is sometimes called the generalised Riemann integral since it follows the definition of the Riemann integral by extending the notion of a the norm of a partition. While there have been attempts to extend the Henstock-Kurzweil integration to more general spaces, these generalisations have lacked the simplicity of the original Henstock-Kurzweil integral for real-valued functions with respect to the Lebesgue measure on $[0,1]$. We note that Using the notion of a set-valued map on non-degenerate closed intervals, it is shown in~\cite{LZ96} that the Henstock-Kurzweil integral with respect to $[0,1]$ can also be defined using a notion of lower and upper Riemann sums.

\section{Partitioning by crescents}\label{d-int}
In this section, we will define a general notion of integral of real-valued functions with respect to a normalised Borel measure $\mu$ on a compact metric space. This so-called $D_\mu$-integral generalises the Lebesgue integral by allowing which only exists if the integrand is absolutely Lebesgue integrable. It thus generalises the Henstock-Kurzweil integral which is only defined for real-valued functions on Euclidean spaces with respect to the Lebesgue measure.

The Henstock-Kurzweil integral of a function $f:[0,1]\to \R$, when its exists, can be evaluated over any subinterval $[a,b]\subset [0,1]$. This situation is like the Riemann integral and is in contrast to the Lebesgue integral of $f$ which can be evaluated with respect to any measurable subset $E\subset [0,1]$.

We will thus seek a family of elementary subsets of $X$ which plays the same role in $D_\mu$-integration as the family of intervals of $[0,1]$ (i.e., open, closed or half-closed intervals) plays in Henstock-Kurzweil integration. We propose that for a general theory of integration on compact metric spaces, the corresponding notion of elementary set is given by a {\em crescent} defined as the intersection of an open set and a closed set. Since $X$ itself is both open and closed, it follows that open sets and closed sets are themselves crescents. 

Next consider the notion of a partition. In the Henstock-Kurzweil theory, integration is performed with respect to the Lebesgue measure on the unit interval in which the boundary of any interval has zero Lebesgue measure. In this case, one works with the notion of a non-overlapping partition of $[0,1]$, i.e., a finite set of closed intervals, with pairwise disjoint interiors, whose union is $[0,1]$. For an arbitrary Borel measure $\mu$, however, the boundary of a crescent, even in the case of intervals of $[0,1]$, can have non-zero $\mu$-measure. We therefore define a {\em partition} of a crescent $C$ as a disjoint union of $C$ by a finite number of crescents. 

In order to have an appropriate set of crescents of $X$ as basic sets for a general integration theory on $X$, we require two properties for our set which hold trivially for intervals of the unit interval:
\begin{itemize}
\item[(i)] The set of intervals of $[0,1]$ is closed under finite intersections.
\item[(ii)] For any two intervals $I_1,I_2\subset[0,1]$, we have the disjoint union $I_1=(I_1\setminus I_2)\cup (I_1\cap I_2)$.

\end{itemize}
We will require (i) and a generalisation of (ii) for the set of crescents we employ in order to develop a general theory of integration on compact metric spaces.

Let ${{\cal B}}$ be a basis of open subsets of $X$ (including the empty set) which is closed under finite (including empty) intersections. The  {\em set of crescents generated by ${\cal B}$} is defined as the set 
\[{\bf C}_{\cal B}=\{A_1\cap {A_2}^c:A_1,A_2\in {\cal B}\}.\]
\begin{definition}\label{part} We say ${\bf C}_{\cal B}$ is {\em partitionable} if the following two conditions hold:
  \begin{itemize}
  \item[(i)]  ${\bf C}_{\cal B}$ is closed under finite intersections.
    \item[(ii)] For any pair of crescents $C,C_0\in {\bf C}_{\cal B}$, there exists a finite collection of pairwise disjoint subsets $S_i\in {\bf C}_{\cal B}$, ($1\leq i\leq m$), satisfying:
      \[C\setminus C_0=S_1\cup S_2\cup\ldots\cup S_m.\]
  
      \end{itemize}
\end{definition}
If ${\bf C}_{\cal B}$ is partitionable, then it satisfies conditions (i) and (ii) of the collection of sets in~\cite[pages 1489-90]{Pfe93}. The following two results (Propositions~\ref{induc} and~\ref{cover-crescent} then follow from properties 1.1, 1.2, 1.3 in~\cite[pages 1490-91]{Pfe93}. We however provide shorter proofs for the collection  ${\bf C}_{\cal B}$ below. 
\begin{proposition}\label{induc}
  
  If ${\bf C}_{\cal B}$ is partitionable and $C,C_i\in {\bf C}_{\cal B}$ for $1\leq i\leq n$, then there exist a finite set of pairwise disjoint crescents $S_j\in {\bf C}_{\cal B}$ for $1\leq j\leq t$ such that
  \[C\setminus\bigcup_{1\leq i\leq n}C_i=\bigcup_{1\leq j\leq t}S_j.\]
  \end{proposition}\begin{proof}
  Using induction on $n\in {\mathbb N}$, the base case for $n=0$ holds by Definition~\ref{part}(ii). Suppose for $n=k-1$, for some $k\geq 1$, the statement holds. Then, there exists a finite collection of pairwise disjoint sets $S_j\in {\bf C}_{\cal B}$ for, say, $1\leq j\leq m$ such that 
  \[C\setminus\bigcup_{1\leq i\leq k}C_i=\left(C\setminus\bigcup_{1\leq i\leq k-1}C_i\right)\setminus C_k\]\[=\left(\bigcup_{1\leq j\leq m} S_j\right)\setminus C_k=\bigcup_{1\leq j\leq m}\left (S_j\setminus C_k\right ),\]
    where the sets $S_j\setminus C_k$ are pairwise disjoint for  $1\leq j\leq m$. Since  ${\bf C}_{\cal B}$ is partitionable, for each set $S_j\setminus C_k$ ($1\leq j\leq m$), there exists a partition by elements of ${\bf C}_{\cal B}$. Thus, the union of these partitions gives the desired partition for $C\setminus\bigcup_{1\leq i\leq k}C_i$.

\end{proof}

We now show that for an abstract compact metric space there is a canonical way to obtain a partitionable set of crescents. Recall that the exterior of a set $S$ is defined as $S^e:=(S^c)^\circ$, which satisfies $(S^e)^c=\overline{S}$ for any subset $S$ of a topological space.
\begin{definition}
We say a basis of open sets of $X$ is {\em full} if it is closed under finite unions, finite intersections and exterior.

  \end{definition}
\begin{proposition}\label{pc}
  Let ${\cal B}$ be a full basis of open subsets of $X$. Then:
	\begin{itemize}
	\item[(i)] For any $A_1,A_2\in {\cal B}$, we have $A\cap\overline{A_2}\in {\bf C}_{\cal B}$.
		\item[(ii)] ${\bf C}_{\cal B}$ is partitionable.
	\end{itemize}
	\end{proposition}	\begin{proof}
	(i) Let $A_3=A_2^e$. Then, we have: $A_1\cap \overline{A_2}=A_1\cap (A_2^e)^c=A_1\cap A_3^c\in {\bf C}_{\cal B}$.
	
	(ii) Let $C,C_0\in {\bf C}_{\cal B}$. Thus, $C=A\cap B^c$ and $C_0=A_0\cap B_0^c$. Then $C\cap C_0=A\cap A_0\cap B^c\cap B_0^c=(A\cap A_0)\cap (B\cup B_0)^c\in {\bf C}_{\cal B}$.

  Moreover,
  \[C\setminus C_0=(A\cap B^c)\cap (A_0^c\cup B_0)=(A\cap B^c\cap A_0^c)\cup (A\cap B^c\cap B_0)\]
  \[=(A\cap (B\cup A_0)^c)\cup \left[(A\cap B^c\cap B_0)\setminus (A \cap (B\cup A_0)^c)\right]\]
  \[=(A\cap (B\cup A_0)^c)\cup\left[(A\cap B^c\cap B_0)\cap (A^c\cup (B\cup A_0))\right]\]
  \[=(A\cap (B\cup A_0)^c)\cup (A\cap A_0\cap B_0\cap B^c)=C_1\cup C_2,\]
  where $C_1=A\cap (B\cup A_0)^c$ and $C_2=A\cap A_0\cap B_0\cap B^c$ are disjoint crescents in ${\bf C}_{\cal B}$. 

\end{proof}

\begin{proposition}\label{cover-crescent}
  If $(U)_{i\in I}$ is an open cover of $\overline{C}$ for $C\in  {\bf C}_{\cal B}$, then there exists a finite number of disjoint sets $C_j\in  {\bf C}_{\cal B}$, with $1\leq j\leq n$, such that $C=\bigcup_{1\leq j\leq n}C_j$ and for each $j\in\{1,\ldots,n\}$ there exists $i\in I$ such that $C_j\subset U_i$.
\end{proposition}\begin{proof}
Take an open cover $(W_t)_{t\in T}$ of $\overline{C}$ by elements of the basis ${\cal B}$ that refines  $(U)_{i\in I}$, i.e. for each $t\in T$, there exists $i\in I$, such that $W_t\subset U_i$. Suppose $W_j$ for $1\leq j\leq n$ is a finite cover of  $\overline{C}$. Put $C_j=C\cap W_j$.  

\end{proof}

The appropriate choice of ${\cal B}$ depends on the context. We give examples of the basic cases one encounters in practice.
\begin{example}\label{unit} Let $X=[0,1]$ and ${\cal B}$ be the set of finite unions of open sub-intervals of $[0,1]$. Then ${\cal B}$ is full and ${\cal C}_B$ is the set of finite unions of all sub-intervals of $[0,1]$ and is partitionable. 
\end{example}
\begin{example}\label{cube} More generally, let $X=\prod_{1\leq i\leq n} [a_i,b_i]\subset \R^n$ be a hyper-rectangle and let ${\cal B}$ be the basis consisting of the finite unions of all open hyper-rectangles. Then ${\cal C}_B$ is the set consisting of finite unions of hyper-rectangles $R$ of $X$ such that some of the $2^n$ boundary faces of $R$ are contained in $R$ and the rest are in $R^c$. It is easy to see that  ${\cal B}$ is full and thus ${\cal C}_B$ is partitionable. 
\end{example}
\begin{example}\label{cantor}
Let $X=\{0,1\}^\omega$ be the Cantor space consisting of elements $x=x_0x_1\ldots$with $x_n\in\{0,1\}$ for $n\in{\mathbb N}$ with the metric defined for $x\neq y$ by $d(x,y)=1/2^n$, where $n\in{\mathbb N}$ is the least natural number such that $x_n\neq y_n$. Take ${\cal B}$ be the basis given by clopen sets, i.e., finite unions of cylinder sets of the form $B_{a_0\ldots a_{n-1}}=\{x\in X:x_i=a_i,0\leq i\leq n-1\}$ for a finite sequence $a_0a_1\ldots a_{n-1}\in \{0,1\}^{n}$. Then  ${\cal B}$ is full and thus ${\bf C}_{\cal B}={\cal B}$ is partitionable. 
  \end{example}

Note that the collection of open intervals of $[0,1]$ is obviously not closed under finite unions (or exterior) and thus the set of all intervals of $[0,1]$ in Example~\ref{unit} cannot be obtained by a basis ${\cal B}$ of open sets of $[0,1]$ satisfying the statement of Proposition~\ref{pc}. Thus, the notion of a partitionable set ${\bf C}_{\cal B}$ of crescents, generated by a full basis ${\cal B}$, allows a unifying theory which contains the fundamental case of Example~\ref{unit} as well as that of Proposition~\ref{pc}.

\section{Partition-gauge pairs and simple valuations}

For the rest of this paper, we fix a basis ${\cal B}$ of open sets that generates a partitionable set ${\bf C}_{\cal B}$ of crescents. In particular, a crescent is always meant to be an element of ${\bf C}_{\cal B}$ and a partition of a crescent $C\in  {\bf C}_{\cal B}$ is always meant to be a partition with respect to ${\bf C}_{\cal B}$.
\begin{definition}
  Let $P=\{{R_i}: 1\leq i\leq n\}$ be a partition of $C\in {\bf C}_{\cal B}$ with pairwise disjoint $R_i\in {\bf C}_{\cal B}$, for $1\leq i\leq n$, and $C=\bigcup_{1\leq i\leq n}{R_i}$. The {\em norm} of $P$ is defined as the real number \[\|P\|:=\max\{\mbox{diam}(R_i): 1\leq i\leq n\}.\]
  \end{definition}

We say a partition $P'$ {\em refines} a partition $P$, written $P\sle P'$, if for all $R'\in P'$, there exists $R\in P$ with $R'\subset R$. It then follows that $P\sle P'$ iff each crescent in $P$ is the union of some crescents in $P'$. Let ${\cal P}$ be the set of all partitions of $X$ and consider the poset $({\cal P},\sle)$.

\begin{proposition}\label{partition-lub}
For two partitions of $X$ given by $P_1=\{R_i:i\in I\}$ and  $P_2=\{S_j:j\in J\}$, their lub is the partition
\[P_1\vee P_2:=\{ R_i\cap S_j:i\in I,j\in J\}.\]\begin{proof}
Since ${\bf C}_{\cal B}$ is closed under finite intersections, $R_i\cap S_j$ is a crescent for each $i\in I,j\in J$. Clearly we have: $\bigcup_{i\in I,j\in J}R_i\cap S_j=X$. 

\end{proof}
 
\end{proposition}
Recall that $O_r(x)$ is the open ball centred at $x\in X$ of radius $r>0$. 

\begin{definition} A {\em gauge} on $\overline{C}$ for $C\in {\bf C}_{\cal B}$  is a strictly positive map $\gamma:\overline{C}\to \R$. 
If $P=\{{R_i}: 1\leq i\leq n\}$ is a partition of $C$ and  $t_i\in {\overline{R_i}}$, for $1\leq i\leq n$, then the set of pairs \[\dot{P}=\{({R_i},t_i):1\leq i\leq n\}\] is called a {\em tagged partition} of $C$. The tagged partition $\dot{P}$ is said to be {\em $\gamma$-fine} on $X$, written $\dot{P}\prec \gamma$, if $\overline{R_i}\subset O_{\gamma(t_i)}(t_i)$ for $1\leq i\leq n$.

\end{definition}

\begin{proposition}\label{pg-exist}
If $\gamma$ is a gauge on $\overline{C}$ for $C\in {\bf C}_{\cal B}$, then there exists a  $\gamma$-fine tagged partition of $C$.

  \end{proposition}
\begin{proof}
Suppose, for a contradiction, that there exists no $\gamma$-fine tagged partition of $C$. Let $P_1=\{R_{1i}: 1\leq i\leq n_1\}$ be a partition of $C_0:=C$ by crescents $R_{1i}\in {\bf C}_{\cal B}$ with $1\leq i\leq n_1$ and norm bounded by $1$. Then there exists $i$ with $1\leq i\leq n_1$ such that there is no $\gamma$-fine tagged partition of the set ${R_{1i}}$. In fact, if each ${R_{1i}}$ for $1\leq i\leq n_1$ has a $\gamma$-fine tagged partition, then the union of these tagged partitions represents a $\gamma$-fine tagged partition for $C_0$ itself. Suppose, therefore, that $C_1:={{R_{1i_1}}}$ for some $i_1$ with $1\leq i_1\leq n_1$ and $\mbox{diam}(C_1)\leq 1$ has no $\gamma$-fine tagged partition. Consider a partition $P_2=\{R_{2i}: 1\leq i\leq n_2\}$ of $C_1$ with norm $1/2$ and let $C_2:=\overline{R_{2i_2}}$ for some $i_2$ with $1\leq i_2\leq n_2$ such that $\overline{R_{2i_2}}$ has no $\gamma$-fine partition. Iteratively, we construct a compact set $C_k\subset X_{k-1}$ with $\mbox{diam}(X _k)\leq 1/k$ that has no $\gamma$-fine tagged partition.  Then the nested set of compact subsets $\overline{C_k}$ with $\mbox{diam}(k)\leq 1/k$ will contain the singleton $\{x_0\}:=\bigcap\overline{ C_{k\geq 0}}$ with $x_0\in \overline{C_0}$. But $\gamma(x_0)>0$ as $\gamma$ is a gauge on $C_0=C$ and thus there exists $k\geq 0$ such that $C_k\subset O_{\gamma(x_0)}(x_0)$ which is a contradiction. 
\end{proof}

Suppose $\mu\in {\bf M}^1(X)$ is a normalised Borel measure on $X$. Since the support of $\mu$ is a closed set, without loss of generality, we assume the support of $\mu$ is $X$. 

The reader is referred to the Appendix for the basic notions in domain theory we will use in the rest of this section. Any partition $P=\{R_i: 1\leq i\leq n\}$ of $X$ induces a normalised simple valuation $\mu_P\in {\bf P}^1({\bf U}(X))$ defined as:
\[\mu_P=\sum_{i=1}^n\mu(R_i)\delta_{\overline{R_i}}\]
We say $\mu_P$ is {\em induced} by the partition $P$. For an open set $O\subset X$ and $\delta>0$, let $O^-_\delta:=\{x\in O: d(x,O^c)>\delta\}$, which is an open set. 

\begin{proposition}\label{partition-net}
  \begin{itemize}
\item[(i)] For any partition $P$: $\mu_P\sle \mu$ in ${\bf P}^1{\bf U}(X)$.
\item[(ii)]  If $P\sle P'$, then $\mu_P\sle \mu_{P'}$.
\item[(iii)]  If $P_i$ for $i\geq 0$ is an increasing sequence of partitions with $\lim_{i\to \infty} \|P_i\|=0$, then $\sup_i{\mu_{P_i}}=\mu$.
\item[(iv)] $\mu=\sup\{\mu_P:P\in {\cal P}\}$.
\end{itemize}
\end{proposition}
\begin{proof}
  (i) Let $O\subset X$ be an open set so that $\Box O\subset {\bf U}(X)$ is a basic Scott open set. Then
  \[\mu_P(\Box O)=\sum_{\overline{R}\subset O}\mu(R)\leq \mu(O),\]
  where the latter step follows from the fact that the crescents of a partition are pairwise disjoint.

  (ii) Easy. 

  (iii) Let $O\subset X$ be any open set and $\epsilon>0$ be given.  Since $O=\bigcup_{n\geq 1}O^-_{1/n}$, we have $\mu(O)=\sup \bigcup_{n\geq 1}\mu(O^-_{1/n})$. Thus, there exists $m\geq 1$ with $\mu(O\setminus \mu(O^-_{1/m})<\epsilon$. Let $i\geq 0$ be such that $\|P_i\|<1/m$. Then $\mu_{P_m}(\Box O)\geq \mu(O^-_{1/m})>\mu(O)-\epsilon$. Since, $\epsilon$ is arbitrary, it follows that $\sup_i{\mu_{P_i}}(O)\geq \mu(O)$, which implies  $\sup_i{\mu_{P_i}}(O)=\mu(O)$ since, by part (i), we already know that $\sup_i{\mu_{P_i}}(O)\leq\mu(O)$.

 (iv) The collection $\{\mu_P:P\in {\cal P}\}$ is a directed set and thus its supremum is well-defined. Since it contains any increasing sequence as in (iii), it follows that its supremum is $\mu$.
  \end{proof}

For a compact subset $C\subset X$ and  $\epsilon>0$, the {\em $\epsilon$-expansion} of $C$ is defined as the compact set, \[C_\epsilon=\{x\in X: \exists y\in C.\, d(x,y)\leq \epsilon\}.\] Given a partition $P=\{R_i:1\leq i\leq n\}$ of $X$ and $\alpha>0$, the {\em $\alpha$-relaxation} of $\mu_P$ is the simple valuation defined as
\[\mu_{P,\alpha}:=\sum_{1\leq i\leq n} \mu(R_i)\delta_{K_i},\] where $K_i:=(\overline{R_i})_\alpha$,
Clearly $\mu_{P,\alpha}\sle \mu_P$ for any $\alpha>0$ and $\mu=\sup_{\alpha>0}\mu_{P,\alpha}$. Finally, for $1>\beta>0,\alpha>0$, we define
\[\mu_{P,\alpha,\beta}=\beta\delta_X+(1-\beta) \mu_{P,\alpha}.\]
\begin{proposition}\label{partition-way-below}
  The collection \[S_\mu=\{\mu_{P,\alpha,\beta}:P\in {\cal P},\alpha>0,\beta>0\}\] is a directed set of normalised simple valuations way-below $\mu$ with supremum $\mu$.
\end{proposition}
\begin{proof}
By the splitting lemma for normalised valuations we have $\mu_{P,\alpha,\beta}\ll \mu_{P,\alpha',\beta'}$ for $0<\alpha'<\alpha$ and $0<\beta'<\beta$. Since, under these assumptions,  $\mu_{P,\alpha,\beta}\ll \mu_{P,\alpha',\beta'}\sle \mu_{P,\alpha'}\sle \mu$, it follows that  $\mu_{P,\alpha,\beta}\ll\mu$. Given two partitions $P_1$ and $P_2$ with $\alpha_1,\beta_1,\alpha_2,\beta_2>0$ and their induced simple valuations $\mu_{P_1,\alpha_1,\beta_1}$ and $\mu_{P_2,\alpha_2,\beta_2}$, let $P=P_1\vee P_2$ with $\alpha=\min\{\alpha_1,\alpha_2\}$ and $\beta=\min\{\beta_1,\beta_2\}$. Then $\mu_{P_1,\alpha_1,\beta_1},\mu_{P_2,\alpha_2,\beta_2}\sle \mu_{P,\alpha,\beta}$, i.e., $S_\mu$ is a directed set of normalised simple valuations way-below $\mu$. On the other hand, for a given $P\in {\cal P}$, the subcollection  $\{\mu_{P,\alpha,\beta}:\alpha>0,\beta>0\}\subset S_\mu$ is directed with supremum $\mu_P$. It follows, by Proposition~\ref{partition-net}(iv), that $\mu=\sup S_\mu$. 

\end{proof}
We next develop some basic order-theoretic properties of tagged partitions and gauges. For two tagged partitions   \[\dot{P}_1=\{(R_{1i},t_{1i}):1\leq i\leq n\}\] \[\dot{P}_2=\{(R_{2j},t_{2j}):1\leq j\leq m\}\] of $C\in {\bf C}_{\cal B}$, we define a partial order by $\dot{P}_1\sle \dot{P}_2$ if $P_1\sle P_2$ and $\{t_{1i}:1\leq i\leq n\}\subset \{t_{2j}:1\leq j\leq m\}$. Let $\dot{{\cal P}}_C$ denote the set of all tagged partitions of $C$. We write $\dot{{\cal P}}:=\dot{{\cal P}}_X$.
\begin{proposition}\label{partition-directed}
The partial order $(\dot{{\cal P}}_C,\sle)$ is a directed set.

\end{proposition}
\begin{proof}
Consider two tagged partitions of $C\in{\bf C}_{\cal B}$: \[\dot{P}_1=\{(R_{1i},t_{1i}):1\leq i\leq n\}\]\[\dot{P}_2=\{(R_{2j},t_{2j}):1\leq j\leq m\}.\] Let \[P=P_1\vee P_2=\{R_{1i}\cap R_{2j}:1\leq i\leq n,1\leq j\leq m\}.\]  For each component $R_{1i}\cap R_{2j}$ with $1\leq i\leq n,1\leq j\leq m$, we proceed as follows. We select a tag $s_{ij}$ for $R_{1i}\cap R_{2j}$ in the first three cases below, whereas for the fourth case $R_{1i}\cap R_{2j}$ is itself partitioned:
\begin{enumerate}
\item If $\{t_{1i},t_{2j}\}\cap R_{1i}\cap R_{2j}=\emptyset$, choose any $s_{ij}\in R_{1i}\cap R_{2j}$.
\item If $\{t_{1i},t_{2j}\}\cap R_{1i}\cap R_{2j}=t_{1i}$, take $s_{ij}:=t_{1i}$.
\item If $\{t_{1i},t_{2j}\}\cap R_{1i}\cap R_{2j}=t_{2j}$, take $s_{ij}:=t_{2j}$.
\item If $\{t_{1i},t_{2j}\}\cap (R_{1i}\cap R_{2j})=\{t_{1i},t_{2j}\}$, with $t_{1i}\neq t_{2j}$, then let $O\in {\cal B}$ with $t_{1i}\in O$ and $t_{2j}\notin O$. Since ${\bf C}_{\cal B}$ is partitionable, there exist a finite set of pairwise disjoint crescents $S_\ell\in {\bf C}_{\cal B}$ ($1\leq \ell\leq k$) such that
  \[(R_{1i}\cap R_{2j})\setminus O=\bigcup_{1\leq \ell\leq k}S_\ell.\] Then there exists a unique $\ell_0\in\{1,\ldots,k\}$ such that $t_{2j}\in S_{\ell_0}$. For $\ell\neq\ell_0$, select $s_{ij\ell}\in S_\ell$ and let $S_0=(R_{1i}\cap R_{2j})\cap O\in {\bf C}_{\cal B}$. Then $\{(S_0,t_{1i}),(S_{\ell_0},t_{2j})\}\cup\{(S_\ell,s_{ij\ell}):\ell\neq \ell_0\}$ is a tagged partition of $R_{1i}\cap R_{2j}$ whose set of tags contains $\{t_{1i},t_{2j}\}$. Taking the union of the tagged partitions obtained as such for all  $R_{1i}\cap R_{2j}$,  with $1\leq i\leq n,1\leq j\leq m$, we obtain a partition which refines $P_1\vee P_2$ and has a set of tags that includes the set $\{t_{1i},t_{2j}: 1\leq i\leq n,1\leq j\leq m\}$ as required.
\end{enumerate}

\end{proof}

\begin{definition} Let $\gamma_C:\overline{C}\to \R$ be a gauge on $\overline{C}$ for $C\in {\bf C}_{\cal B}$.
A pair $(\dot{P}_C,\gamma_C)$ with $\dot{P}_C\prec \gamma_C$ is called a {\em PG pair} for $C$. The set of PG pairs for $C$ is denoted by $\dot{{\cal P}}{\cal G}_C$. 
\end{definition}
We write  $(\dot{P},\gamma):=(\dot{P}_X,\gamma_X)$ and $\dot{{\cal P}}{\cal G}:=\dot{{\cal P}}{\cal G}_X$.
Next we define a partial order on PG pairs by putting $(\dot{P}_C,\gamma_C)\sle (\dot{P'_C},\gamma'_C)$ if $\dot{P}_C\sle 
\dot{P'}_C$ and $\gamma_C\geq \gamma'_C$.

\begin{lemma}\label{exist-tag}
Let $Q$ be a partition of $C\in {\bf C}_{\cal B}$ and let $T\subset \overline{C}$ be a finite set of points. Then, given any gauge $\gamma_C$ on $\overline{C}$, there exists a tagged partition $\dot{P}_C$ of $C$ such that: (i) $P$ refines $Q$, (ii) the set of tags of ${\dot{P}}_C$ contains $T$, and (iii) $\dot{P}_C\prec \gamma_C$.

\end{lemma}\begin{proof}
  Assume $Q=\{R_i\in {\bf C}_{\cal B}:1\leq i\leq n\}$. Let $r>0$ be small enough such that $r<\min_{t\in T}\gamma_C(t)$. For each $t\in T$, let $B_t\in {\cal B}$ with $B_t\subset O_r(t)$ be an open set with $t\in B_t$. Put $O:=\bigcup_{t\in T} B_t$. By Proposition~\ref{induc}, for each $i$ with $1\leq i\leq n$, the set ${R_i}\setminus O$ is the disjoint union of elements in ${\bf C}_{\cal B}$. By Proposition~\ref{pg-exist}, each of these disjoint sets has a $\gamma$-fine tagged partition. By taking the union of these $\gamma$-fine partitions, we obtain a $\gamma$-fine partition, say, $\dot{P_i}=\{(S_{ij},s_j):1\leq j\leq i_j\}$ of ${R_i}\setminus O$. Then $\dot{P}=(\bigcup_{1\leq i\leq n}\dot{P_i})\cup \{(O_r(t),t):t\in T\}$ is a tagged partition of $C$ that satisfies the three required properties (i),(ii) and (iii). 
\end{proof}
We can now show that the set of PG pairs is a directed set.

\begin{proposition}\label{pg-directed}
The partial order $\dot{\cal P}{\cal G}_C$, for $C\in {\bf C}_{\cal B}$, is a directed set.
\end{proposition}\begin{proof}
Let $(\dot{P_1}, \gamma_1)$ and $(\dot{P_2}, \gamma_2)$ be two PG pairs for $C$ and assume $\dot{P_1}$ and $\dot{P_2}$ have set of tags $T_1$ and $T_2$ respectively. Let $Q=P_1\vee P_2$ and $\gamma=\min\{\gamma_1,\gamma_2\}$. By Lemma~\ref{exist-tag}, there exists a tagged partition $\dot{P}$ with set of tags $T$  such that $Q\sle P$, $T_1\cup T_2\subset T$ and $\dot{P}\prec \gamma$, i.e., $(\dot{P_1}, \gamma_1),(\dot{P_2}, \gamma_2)\sle (\dot{P}, \gamma)$.
\end{proof}

For each PG pair $(\dot{P},\gamma)$, with $\dot{P}=\{(R_i,t_i):1\leq i\leq n\}$, the normalised measure $\mu$ induces a simple valuation 
\[\mu_{\dot {P},\gamma}=\sum_{1\leq i\leq n} \mu(R_i)\delta_{t_i}\in\M^1(X)\]

\begin{theorem}~\label{port}{\bf Portmanteau}~\cite[p. 372]{Str11}
  Let $Y$ be a metric space and $\{\mu_k:k\in A\}$  a net in ${\bf M}^1(Y)$. Then the following conditions are equivalent for any $\mu\in  {\bf M}^1(Y)$.
  \begin{itemize}
  \item[(i)] $\lim_{k\in A}\mu_k=\mu$ in the weak topology of ${\bf M}^1(Y)$.
  \item[(ii)] For every open set $O\subset E$, we have: $\lim\inf_{k\in A}\geq \mu(O)$.
    \item[(iii)] For every bounded function $f$ that is continuous almost everywhere with respect to $\mu$, we have: $\lim_{k\in A} \int f\,d\mu_k=\int f \,d\mu$.

  \end{itemize}
  \end{theorem}
  
Using Portmanteau theorem, we can now deduce the following:
\begin{theorem}\label{net-weak}
The net $\{\mu_{\dot{P},\gamma}:(\dot{P},\gamma)\in \dot{\cal P} {\cal G}\}\subset\M^1(X)$ converges in weak topology to $\mu$, i.e.,
\[\lim_{(\dot{P},\gamma)\in \dot{\cal P} {\cal G}} \mu_{\dot{P},\gamma}=\mu\]
\end{theorem}\begin{proof}
Let $O\subset X$ be any non-empty open set. Since $O$ is the union of increasing open sets $W$ with $\overline{W}\subset O$, for each $\epsilon>0$, there exists $W$ such that $\mu(W)>\mu(O)-\epsilon$ and $\overline{W}\subset O$. Let $\delta>0$ be such that $\overline{W}_\delta\subset O$ and consider the gauge $\gamma$ with $\gamma(x)=\delta/2$ for all $x\in X$. Take $\dot{P}\prec \gamma$. If $R$ is a crescent of $P$ and $R\cap O^c\neq \emptyset$, then $R\cap \overline{W}=\emptyset$. It follows that $\mu_{\dot{P},\gamma}(O)\geq \mu(W)>\mu(O)-\epsilon$ and this also holds for all $\mu_{\dot{P'},\gamma'}(O)$ with $(\dot{P},\gamma)\sle (\dot{P'},\gamma')$. Since $\epsilon>0$ is arbitrary, we get $\lim\inf \mu_{\dot{P},\gamma}(O)\geq \mu(O)$. The result now follows from Theorem~\ref{port}. 
\end{proof}

Since the relative Scott topology on the subset of maximal elements ${\bf M}^1(X) \subset {\bf P}^1{\bf U}(X)$ coincides with the weak topology, Theorem~\ref{net-weak} is the analogue of the convergence in the Scott topology to $\mu$ of the simple valuations $\nu\in {\bf P}^1{\bf U}X$ way below $\mu$:
\[\sup_{\nu\ll \mu} \nu=\mu.\]

We will further show below that, for integration of functions that are continuous on $X$, these two structures, namely the directed set of normalised simple valuations in ${\bf P}^1{\bf U}X$ way-below $\mu$ and the directed set of PG pairs of $X$, can be equally employed.  Let $f:X\to \R$ be a continuous function and consider a Borel measure $\mu$ on the compact metric space $X$. We show in the equivalence of (ii) and (iii) in the following theorem that the idea of a gauge arises naturally in the domain-theoretic derivation of the Lebesgue integral of $f$ using simple valuations way-below $\mu$.

\begin{theorem}\label{simple-val-gauge-eq}
  Given a real-valued continuous function $f$ on $X$,  the  following are equivalent for $r\in \R$:
  \begin{itemize}
  \item[(i)] The Lebesgue integral of $f$ with respect to the normalised Borel measure $\mu$ satisfies $\int f\,d\mu=r$.
  \item[(ii)] For each $\epsilon>0$, there exists a partition $P$ and $1>\alpha,\beta>0$ with $r\in \int f\mu_{P,\alpha,\beta}$ and $\mbox{diam}(\int f\mu_{P,\alpha,\beta})<\epsilon$.
  \item[(iii)] For each $\epsilon>0$, there exists a gauge $\gamma$ on $X$ such that $|\int f\,d\mu_{\dot{P},\gamma}-r|<\epsilon$ for any $\gamma$-tagged partition $\dot{P}$.
  \end{itemize}
  \end{theorem}

\begin{proof}

  (i)$\,\iff\,$(ii). By~\cite[Theorem 6.5]{Eda94}, $f$, being continuous, is R-integrable with respect to $\mu$, and, by~\cite[Theorem 7.2]{Eda94}, the Lebesgue integral of $f$ is equal to its R-integral. Recall, by Proposition~\ref{partition-way-below} that $\mu_{P,\alpha,\beta}\ll \mu$ for $P\in{\cal P}$ and $1>\alpha,\beta>0$. The equivalence of (i) and (ii) now follows since, by~\cite[Propositions 4.6 and 4.8]{Eda94}, $f$ is R-integrable iff (ii) holds. \\

  (ii)$\,\Rightarrow\,$(iii). Suppose $\mu_{P,\alpha,\beta}$, for some \[P=\{R_i:1\leq i\leq n\}\in{\cal P},\] and $1>\alpha,\beta>0$, satisfies (ii) for some $r\in \R$ and $\epsilon>0$. Since  $\mu_{P,\alpha,\beta}\sle \mu_{P,\alpha}$, it follows, by~\cite[Proposition 4.2]{Eda94}, that $S^\ell(f,\mu_{P,\alpha,\beta})\leq S^\ell(f,\mu_{P,\alpha})$ and 
  $S^u(f,\mu_{P,\alpha})\leq S^u(f,\mu_{P,\alpha,\beta})$. Hence, \begin{equation}\label{refine-P}\int f\,d\mu_{P,\alpha,\beta}\sle \int f\,d\mu_{P,\alpha}.\end{equation} We will now use the simple valuation $\mu_{P,\alpha}$ to define a gauge $\gamma$ on $X$.  Let $K_i=(\overline{R_i})_\alpha$. Then the collection $\{K^\circ_i:1\leq i\leq n\}$ gives an open cover of $X$, i.e., $X=\bigcup_{1\leq i\leq n}K^\circ_i$. For $x\in X$, let $\gamma(x)=d(x,\partial R_i)$ if $x\in R_i^\circ$ for some $i$ (with $1\leq i\leq n$), which would be unique since the crescents $R_i$'s are pairwise disjoint for $1\leq i\leq n$. Otherwise, if $x\notin\bigcup_{1\leq i\leq n}R^\circ_i$,  let \[\gamma(x):=\sup\{r>0: x\in K^\circ_i,\ C_r(x)\subset K_i,1\leq i\leq n\}.\]Then $\gamma(x)>0$ for all $x\in X$. Suppose now \[\dot{P'}=\{(R'_i,t_i):1\leq i\leq m\}\] is any $\gamma$-fine tagged partition of $X$. We claim that $\mu_{P,\alpha}\sle \mu_{P'}$. Recall that we have \[\mu_{P,\alpha}=\sum_{1\leq i\leq n} \mu(R_i)\delta_{K_i}\qquad\mu_{P'}=\sum_{1\leq j\leq m} \mu(R'_j)\delta_{\overline{R'_j}}\] where $K_i=(\overline{R_i})_\alpha$. Then, by construction of $\gamma$, for each $j$ with $1\leq j\leq m$, there exists $i$ with $1\leq i\leq n$ such that ${R'_j}\subset K_i$, and ${R'_j}\cap K_i\neq \emptyset$ implies  ${R'_j}\subset K_i$.  Define non-negative real numbers $t_{ij}$ for $1\leq i\leq n$ and $1\leq j\leq m$ as follows:
	\[t_{ij}=\mu(R_i\cap R'_j)\]
Then $t_{ij}$, for $1\leq i\leq n$ and $1\leq j\leq m$, satisfies 
\[\sum_{1\leq j\leq m}t_{ij}=\mu(R_i),\quad\sum_{1\leq i\leq n}t_{ij}=\mu(R'_j)\]
 and $t_{ij}\neq 0$ implies ${R'_j}\subset K_i$, which shows by the Splitting lemma for normalised simple valuations~\cite[Proposition 3.1]{Eda94} that $\mu_{P,\alpha}\sle \mu_{P'}$. It follows, by Equation~(\ref{refine-P}), that $\int f\,d\mu_{P,\alpha}\sle \int f\,d\mu_{P'}$ and thus $|\int f\,d\mu_{\dot{P'},\gamma}-r|<\epsilon$.\\

(iii)$\,\Rightarrow\,$(ii). Let $M>0$ be a bound for $f$, i.e., $|f(x)|\leq M$ for $x\in X$ and let $\epsilon>0$ be given. By assumption, there exists a gauge $\gamma$ on $X$ such that $|\int f\,d\mu_{\dot{P},\gamma}-r|<\epsilon/3$ for any tagged partition $\dot{P}\prec \gamma$. Since $f$ is continuous, and thus uniformly continuous, on the compact set $X$, there exists $\delta>0$ such that $|f(x)-f(y)|<\epsilon/3$ whenever $x,y\in X$ and $d(x,y)<\delta$. Let the  gauge $\gamma'$ be defined for $x\in X$ by: 
\[\gamma'(x)= \min\{\gamma(x),\delta/2\}.\]

Then $0<\gamma'(x)\leq \gamma(x)$ for $x\in X$. Suppose now \[\dot{P}=\{(R_i,x_i):1\leq i\leq n\}\] is any $\gamma'$-fine tagged partition, thus also $\gamma$-fine. We claim that $\mu_{P,\delta/2,\epsilon/6M}$ satisfies (ii).  It suffices to show that \[|S_\xi(f,\mu_{P,\delta/2,\epsilon/6M})-r|<\epsilon,\] for any selection of $\xi$ with $y_0:=\xi(X)\in X$ and $y_i:=\xi(R_i)\in (\overline{R_i})_{\delta/2}$ for $1\leq i\leq n$. Since $\dot{P}\prec \gamma'$, we have $\overline{R_i}\subset O_{\gamma'(x_i)}$ for $1\leq i\leq n$. From $y_i\in (\overline{R_i})_{\delta/2}$, it follows that $d(x_i,y_i)<\delta$ for $1\leq i\leq n$.
We have, for $\xi':=\xi\restriction P_{\mu,\delta/2}$:
\[\left|S_\xi(f,,\mu_{P,\delta/2,\epsilon/6M})-S_{\xi'}(f,\mu_{P,\delta/2})\right|\]\[\leq\left|\frac{\epsilon}{6M}\left(f(y_0)\mu(X)-\sum_{1\leq i\leq n}f(y_i)\mu(R_i)\right)\right|=\frac{\epsilon}{3},\]
since $\mu(X)=1$ and $f$ is bounded by $M$.

Next we estimate:
\[\left|S_{\xi'}(f,\mu_{P,\delta/2})-S(f,\dot{P},\mu)\right|=\]\[
\left|\sum_{1\leq i\leq n} (f(y_i)-f(x_i))\mu(R_i)\right|<\epsilon/3,\]
since $|f(y_i)-f(x_i)|<\epsilon/3$. Finally, since $\dot{P}\prec\gamma$ we have:
\[\left|\int f\,d\mu_{\dot{P},\gamma}-r\right|<\epsilon/3.\]

Putting the three estimates above together we obtain:

\[\left|S_\xi(f,\mu_{P,\delta/2,\epsilon/6M})-r\right|\]
\[\leq\left|S_\xi(f,\mu_{P,\delta/2,\epsilon/6M})-S_{\xi'}(f,\mu_{P,\delta/2})\right|\]
\[+\left|S_{\xi'}(f,\mu_{P,\delta/2})-S(f,\dot{P},\mu)\right|+\left|\int f\,d\mu_{\dot{P},\gamma}-r\right|\]
\[< \frac{\epsilon}{3}+\frac{\epsilon}{3}+\frac{\epsilon}{3}=\epsilon.\]
\end{proof}

We  note that, in the present case when $f$ is assumed continuous, in Theorem~\ref{simple-val-gauge-eq}, the existence of the gauges in (iii) for the Lebesgue integrability of $f$ with respect to the normalised measure $\mu$ is derived using the family of normalised simple valuations in (ii), which form a directed set of simple valuations way-below  $\mu$. More specifically, the gauge satisfying (iii) was constructed from a normalised simple valuation satisfying (ii) and vice versa. One can say that normalised simple valuations way-below $\mu$ played the role of first class objects from which gauges were derived.

Combining Portmanteau theorem with Theorem~\ref{net-weak}, we conclude:
\begin{corollary}
For every continuous function $f:X\to \R$,
we have: \[\lim_{(\dot{P},\gamma)\in \dot{\cal P} {\cal G}} \int f\,d\mu_{\dot{P},\gamma}=\int f\,d\mu\]
  \end{corollary}

\section{Generalised Henstock-Kurzwel Integration}\label{unbounded}
 
We now have a rich directed family of PG pairs which will play the same pivotal role for the integration of general functions as the directed set of normalised simple valuations way-below a given normalised Borel measure for the integration of continuous functions. We can now define the $D_\mu$-integrability of any function. 
\begin{definition}
We say $f:\overline{C}\to \R$ has a $D_\mu$-integral $\int_C f\,d\mu=r$ with respect to a normalised Borel measure $\mu$ if there is a sub-net $\dot{\cal P}{\cal G}_C(f)\subset \dot{\cal P}{\cal G}_C$ such that $\lim_{(\dot{P},\gamma)\in\dot{\cal P}{\cal G}_C(f)}\int f\,\mu_{\dot{P},\gamma}=r$. 
\end{definition}
Note that $f$ has to be defined on $\overline{C}$, required in the proof of Proposition~\ref{pg-exist}, for the integral $\int_C f\,d\mu$ to be defined. Since limits of nets are unique in Hausdorff spaces, the  integral, if it exists, is well-defined. We immediately have:

\begin{proposition}\label{HK}  $f:\overline{C}\to \R$, for $C\in {\bf C}_{\cal B}$, is $D_\mu$-integrable with respect to $\mu$ on $C$ with value $r$  iff for all $\epsilon>0$ there exists a gauge $\gamma$  on $\overline{C}$ such that for any tagged partition $\dot{P}\prec \gamma$ we have: $|\int_C f\,\mu_{\dot{P},\gamma}-r|<\epsilon$ .
\end{proposition}
As usual, we write: $\int f\,d\mu:=\int_Xf\,d\mu$ when the latter exists. 
If we let $X=[0,1]$ with $\mu$ the Lebesgue measure on $[0,1]$, then Proposition~\ref{HK} shows that the $D_\mu$-integral and the Henstock-Kurzweil integral coincide. 

We now derive some basic properties of the  $D_\mu$-integral.  In~\cite{LZ96}, it is shown using the notion of contractions of intervals that the classical HK integral can be obtained with lower and upper sums as in the case of the Riemann integral. Here, we show similarly that the more general $D_\mu$-integral can be obtained from lower and upper sums by simply using the standard gauge. For any $f:X\to \R$, the lower and upper sums are defined as extended real numbers:

\[S^\ell(f,\gamma,\mu)=\inf \left\{\int f\,d\mu_{\dot{P},\gamma}: \dot{P}\prec \gamma\right\}\]
\[S^u(f,\gamma,\mu)=\sup \left\{\int f\,d\mu_{\dot{P},\gamma}: \dot{P}\prec \gamma\right\}\]
Note that for $\gamma_1\leq \gamma_2$ the relation $\dot{P}\prec \gamma_1$ implies $\dot{P}\prec \gamma_2$ and hence $S^\ell(f,\gamma_1,\mu)\geq S^\ell(f,\gamma_2,\mu)$ and $S^u(f,\gamma_1,\mu)\leq S^u(f,\gamma_2,\mu)$. We can now define the lower and upper integrals of $f$ with respect to $g$:
\[L\int_a^bf\,d\mu:=\lim_{\gamma}S^\ell(f,\gamma,\mu)=\sup_{\gamma}S^\ell(f,\gamma,\mu)\]
\[U\int_a^bf\,d\mu:=\lim_{\gamma}S^u(f,\gamma,\mu)=\inf_{\gamma}S^u(f,\gamma,\mu)\]
 as extended real numbers. By definition of liminf and limsup we have:
\begin{corollary}
\[L\int f\,d\mu=\liminf_{{\dot P}\prec\gamma}\int f\,d\mu_{\dot{P},\gamma}\]
\[U\int f\,d\mu=\limsup_{{\dot P}\prec \gamma} \int f\,d\mu_{\dot{P},\gamma}\]

\end{corollary}

We now obtain a new equivalent definition for the $D_\mu$-integrability of $f$.
\begin{corollary}\label{lower-upper-integ}
A function $f:X\to \R_\bot$ is $D_\mu$-integrable iff $L\int f\,d\mu=U\int f\,d\mu$.

\end{corollary}
\begin{proof} Suppose $f$ is $\mu$-integrable to $r$. Let $\epsilon>0$ be given and let $\gamma$ be a gauge such that $\dot{P}\prec \gamma$ implies $|\int f\,d\mu_{\dot{P},\gamma}-r|< \epsilon$. We obtain $\liminf_{\dot P\prec \gamma}\int f\,d\mu_{\dot{P},\gamma}=\limsup_{\dot P\prec \gamma}\int f\,d\mu_{\dot{P},\gamma}=r$. Suppose on the other hand $\liminf_{\dot{P}\prec \gamma}S(f,\dot{P},\gamma,g)=r$ and $\epsilon>0$ is given. Then $\sup_{\gamma}S^\ell(f,\gamma,\mu)=\inf_{\gamma}S^u(f,\gamma,\mu)=r$ and hence there exists a gauge $\gamma$ with $|S^\ell(f,\gamma,g)-r|\leq \epsilon$ and  $|S^u(f,\gamma,g)-r|\leq \epsilon$ which implies $|\int f\,d\mu_{\dot{P},\gamma}-r|< \epsilon$ for all $\dot{P}\prec \gamma$.

\end{proof}

\begin{proposition}\label{basic-int-prop}
\begin{itemize}
\item[(i)] If $f:X\to \R$ is zero almost everywhere with respect to $\mu$ then $\int f\,d\mu=0$.
\item[(ii)] If $f_1$ and $f_2$ are $D_\mu$-integrable then so is $f_1+f_2$ with $\int (f_1+f_2)\,d\mu=\int f_1\,d\mu+\int f_2\,d\mu$.
\item[(iii)] If $c\in \R$ and $f$ is $D_\mu$-integrable then so is $cf$ with $\int cf\,d\mu=c\int f\,d\mu$.
  \item[(iv)] If $f\geq 0$ is $D_\mu$-integrable then $\int f\,d\mu\geq 0$. 
\end{itemize}
\end{proposition}\begin{proof}
(i) Let $\epsilon>0$ be given and let $D$ be the null set where $f$ is non-zero. Then $D=\bigcup_{n\geq 1}D_n$ where $D_n=\{x: n> |f(x)|\geq n-1\}$ are disjoint sets and $\mu(D_n)=0$ for $n\geq 1$. Then there exist  open balls $W_{nm}$, for $m\geq 1$, with $\sum_{m\geq 1}\mu(W_{nm})<\epsilon/n2^{n}$ such that $D_n\subset \bigcup_{m\geq 1}W_{nm}$. We now define a gauge $\gamma$ on $X$. If $x\notin D$ let $\gamma(x)=1$. If $x\in D$, then there exists a unique positive integer $n_x$ such that $x\in D_{n_x}$; take the least integer $m_x\geq 1$ such that $x\in W_{n_xm_x}$. Take $\gamma(x)>0$ small enough such that $O_{\gamma(x)}(x)\subset W_{n_xm_x}$. If $\dot{P}\prec \gamma$, then, since $|f|$ is bounded by $n$ in $D_n$, we obtain: \[\left|\int f\,d\mu_{\dot{P},\gamma}\right|< \sum_{n\geq 1} n\epsilon/(n2^n)=\epsilon.\]

(ii) Let $\epsilon>0$ be given. By $D_\mu$-integrability of $f_1$ and $f_2$, say to values $r_1$ and $r_2$ respectively, there exist gauges $\gamma_1$ and $\gamma_2$ such that for all tagged partitions $\dot{P_i}\prec \gamma_i$ with $i=1,2$, we have $|\int f_i\,d\mu_{\dot{P_i},\gamma_i}-r_i|<\epsilon/2$ for $i=1,2$. Put $\gamma=\min\{\gamma_1,\gamma_2\}$. If $\dot{P}\prec \gamma$, then 
\[\left|\int (f_1+f_2)\,d\mu_{\dot{P},\gamma}-r_1+r_2\right|\]\[\leq \left|\int f_1\,d\mu_{\dot{P_1},\gamma_1}-r_1\right|+\left|\int f_2\,d\mu_{\dot{P_2},\gamma_2}-r_2\right|<\epsilon/2+\epsilon/2=\epsilon.\]

(iii)-(iv) Straightforward.
\end{proof}

We have the usual {\em Cauchy condition} for  $D_\mu$-integrability, whose straightforward proof is skipped:

\begin{proposition}\label{cauchy}
A function $f:\overline{C}\to \R$, for $C\in {\bf C}_{\cal B}$, is $D_\mu$-integrable iff it satisfies the Cauchy condition: for each $\epsilon>0$ there exists a gauge $\gamma$ on $\overline{C}$ such that for any two PG pairs $(\dot{P_1},\gamma),(\dot{P_2},\gamma)\in \dot{{\cal P}}{\cal G}_C$  we have $|\int f\mu_{\dot P_1,\gamma}-
\int f\mu_{\dot P_2,\gamma}|<\epsilon$. 
\end{proposition}
Using the Cauchy condition, we can now show that the  $D_\mu$-integral is additive.
\begin{proposition}\label{restrict-int}
  \begin{itemize}
  \item[(i)] If $R,C\in {\bf C}_{\cal B}$ with $C\subset R$ and $f:R\to \R$ is $D_\mu$-integrable then so is the restriction $f:{C}\to \R$.
  \item[(ii)] Suppose $P=\{R_i: 1\leq i\leq n\}$ is a partition of $R\in{\bf C}_{\cal B}$. Then $f:R\to \R$ is  $D_\mu$-integrable on $R$ iff $f$ is  $D_\mu$-integrable on $R_i$ for $ 1\leq i\leq n$ and, in which case,
    \[\int_R f\,d\mu=\sum_{i=1}^n \int_{R_i}f\,d\mu.\]

    \end{itemize}
\end{proposition}
\begin{proof}
  (i) Let $\epsilon>0$ be given and let $\gamma$ be a witness for the Cauchy condition on the $D_\mu$-integrability of $f:R\to \R$. Consider two tagged partitions $\dot{Q_j}\prec \gamma\restriction_C$ for $j=1,2$. Since ${\bf C}_{\cal B}$ is partitionable, there exist a finite number of disjoint sets
  $C_i\in {\bf C}_{\cal B}$ ($1\leq i\leq n$) such that $R\setminus C=\bigcup_{1\leq i\leq n}C_i$. Let $\dot{P_i}\prec \gamma\restriction_{\overline C_i}$ be a tagged partition of $C_i$ for $1\leq i\leq n$. Then \[\dot{Q_1'}:=\dot{Q_1}\cup\left(\bigcup_{1\leq i\leq n}\dot{P_i}\right)\qquad\dot{Q'_2}:=\dot{Q_2}\cup\left(\bigcup_{1\leq i\leq n}\dot{P_i}\right)\] are two $\gamma$-fine tagged partitions of $R$. Thus,
  \[\left|\int f\,\mu_{\dot{Q_1},\gamma}-\int f\,\mu_{\dot{Q_2},\gamma}\right|=\left|\int f\,\mu_{\dot{Q'_1},\gamma}-\int f\,\mu_{\dot{Q'_2},\gamma}\right|<\epsilon,\]
  which shows that the Cauchy condition for $D_\mu$-integrability of $f$ on $C$ is satisfied.

  (ii) Suppose $f$ is $D_\mu$-integrable on $R_i$ with  $\int_{R_i}f\,d\mu=r_i$ for $ 1\leq i\leq n$. Let $\epsilon>0$. There exist gauges $\gamma_i:\overline{R_i}\to \R$ such that $\dot{P_i}\prec \gamma_i$ implies $|\int f\,d\mu_{\dot{P_i},\gamma_i}-r_i|<\epsilon/n$ for $ 1\leq i\leq n$. Define a gauge $\gamma:R\to \R$ as follows. For $x\in \overline{R}$, let
  \[\gamma(x)=\left\{\begin{array}{ll} \min\{\gamma_i(x),d(x,\partial R_i)\}&\mbox{if }\exists i.\,x\in R_i^\circ\\
  \min\{\gamma_i(x): x\in \partial R_i,1\leq i\leq n\}&\mbox{otherwise}\\\end{array}\\\right.\]
Suppose now we have a tagged $\gamma$-fined partition of $R$ given by: $\dot{P}=\{(C_j,t_j):1\leq j\leq m\}\prec \gamma$. If $t_j\in R_i^\circ$ for some $i\in\{1,\ldots,n\}$ then, by construction of $\gamma$, we have $C_j\subset R_i$. Otherwise for $i\in\{1,\ldots,n\}$  such that $t_j\in\partial R_i$ and $R_i\cap C_j\neq \emptyset$, let $C_{ij}=R_i\cap C_j$. Then, for $i\in\{1,\ldots,n\}$, let
\[\dot{P'_i}=\{(C_j,t_j):t_j\in R_i^\circ,1\leq j\leq m\}\]\[\cup\{(C_{ij},t_j):t_j\in \partial R_i\,\&\,R_i\cap C_j\neq\emptyset,1\leq j\leq m\}.\]
Then $\dot{P'_i}$ is a $\gamma_i$-fine is partition of $R_i$. Let $\dot{P'}=\bigcup_{1\leq i\leq n}\dot{P'_i}$. Since $R_i=\bigcup_{j=1}^mR_i\cap C_j$ for $1\leq i\leq n$, we have:
\[\left|\int f\,d\mu_{\dot{P},\gamma}-\sum_{i=1}^nr_n\right|=\left|\sum_{i=1}^n\int f\,d\mu_{\dot{P'_i},\gamma_i}-r_i\right|\]
\[\leq \sum_{i=1}^n\left|\int f\,d\mu_{\dot{P'_i},\gamma_i}-r_i\right|<\sum_{i=1}^n\epsilon/n=\epsilon,\]
and thus $f$ is $D_\mu$-integrable on $R$. On the other hand if $f$ is $D_\mu$-integrable on $R$, then by part (i) it is $D_\mu$-integrable on $R_i$ for $ 1\leq i\leq n$ with $\int_{R_i}f\,d\mu=r_i$, say. By the earlier proof, we then obtain the desired equality for the additivity of the $D_\mu$-integral.

\end{proof}

\begin{proposition}\label{charac-function}
The characteristic function $\chi_E$ of a measurable set $E$ is $D_\mu$-integrable with $\int \chi_E\,d\mu=\mu(E)$.
\end{proposition}\begin{proof}
  The case $\mu(E)=0$ is already proved in Proposition~\ref{basic-int-prop}(i). Assume $\mu(E)>0$ and $\epsilon>0$. By regularity of $\mu$, there exist open set $O$ and closed set $C$ such that $E\subset O$ with $\mu(E)>\mu(O)-\epsilon$ and $C\subset E$ with $\mu(E)<\mu(C)-\epsilon$. Define a gauge $\gamma$ on $X$ as follows:
  \[\gamma(x)=\left\{\begin{array}{ll} d(x,\partial{O})& x\in C\\
  d(x,C)&x\in O^c\\
  \min\{d(x,C),d(x,\partial{O})\}& x\in O\setminus C\\\end{array}\right.\]
  If $\dot{P}\prec \gamma $ is a tagged partition of $X$, then
  \[\mu(O)\geq \int \chi_E\,d\mu_{\dot{P},\gamma}\geq \mu(C),\qquad\mbox{i.e.,}\]
  \[\mu(E)+\epsilon> \int \chi_E\,d\mu_{\dot{P},\gamma}> \mu(E)-\epsilon\]
  as required.
\end{proof}
From Proposition~\ref{restrict-int}(ii), we obtain:
\begin{corollary}
  If $f=\sum_{1\leq i\leq n}c_i\chi_{E_i}$ is a step function, then
  \[\int f\,d\mu=\sum_{1\leq i\leq n}c_i\mu(E_i).\]
\end{corollary}

The following lemma extends that of Saks-Henstock for the case of Lebesgue measure on $[0,1]$ to  arbitrary normalised measure on a compact metric space. A {\em sup-partition} of $X$ is a finite set of disjoint crescents in ${\bf C}_{\cal B}$. If $\gamma$ is a gauge, then $\{(R_i,t_i):1\leq i\leq n\}$ is a  {\em $\gamma$-fine tagged sub-partition} of $X$ if  $\{R_i:1\leq i\leq n\}$ is a subpartition, $t_i\in\overline{R_i}$ and $R_i\subset O_{\gamma(t_i)}(t_i)$ for $1\leq i\leq n$

\begin{lemma}\label{saks}
  Suppose $f:X\to\R$ is $D_\mu$-integrable and, for $\epsilon>0$, let $\gamma$ be a gauge such that $\dot{P}\prec \gamma$ implies:
  \[\left|\int f\,d\mu_{\dot P,\gamma}-\int f\,d\mu\right|\leq \epsilon.\]
  Then for any $\gamma$-fine tagged sub-partition \[\dot{Q}=\{(R_i,t_i):1\leq i\leq n\}\] of $X$ with $R:=\bigcup_{1\leq i\leq n}R_i$, we have:
  \[\left |\int f\,d\mu_{\dot{Q},\gamma}-\int_{R} f\,d\mu\right|\leq \epsilon.\]

  \end{lemma}

\begin{proof}
  By Proposition~\ref{induc}, there exists a finite set of pairwise disjoint crescents $C_j\in {\bf C}_{\cal B}$, with say $1\leq j\leq m$, such that $X\setminus (\bigcup_{1\leq i\leq n}R_i)=\bigcup_{1\leq j\leq m}C_j$. By Proposition~\ref{restrict-int}(i), $f$ is $D_\mu$-integrable on each $C_j$ for $1\leq j\leq m$. Let $a>0$ and assume $\gamma_j:\overline{C_j}\to \R$ be a gauge with $\gamma_j\leq \gamma\restriction_{C_j}$ such that $\dot{Q_j}\prec \gamma_j$ and we have $|\int_{C_j} f\,d\mu-\int f\,d\mu_{\dot{Q_j},\gamma_j}|<a/m$  for $1\leq j\leq m$. Put $\dot{P}=\dot{Q}\cup (\bigcup_{1\leq j\leq m}\dot{Q_j})$. Then, $\dot{P}\prec \gamma$ and we have:
  \[\int f\,d\mu_{\dot{P}}=\int f\,d\mu_{\dot{Q},\gamma}+\sum_{1\leq j\leq m}\int f\,d\mu_{\dot{Q_j},\gamma_j}\]
      \[\int f\,d\mu=\int_Rf\,d\mu+\sum_{1\leq j\leq m}\int_{C_j}f\,d\mu.\]

 Consequently,
 \[ \left|\int f\,d\mu_{\dot{Q},\gamma}-\int_{R} f\,d\mu\right|\leq \left|\int f\,d\mu_{\dot{P}}-\int f\,d\mu\right|+\sum_{j=1}^m\left|\int f\,d\mu_{\dot{Q_j},\gamma_j}
   - \int_{C_j}f\,d\mu. \right|<\epsilon+m(a/m)=\epsilon+a\]
     Since $a>0$ is arbitrary, the result follows. 

\end{proof}

Based on Lemma~\ref{saks}, we can then prove the following main result. 
\begin{theorem}\label{mct} ({\bf Monotone Convergence.}) Let $(f_n)_{n\geq 0}$ be a monotone sequence of $D_\mu$-integrable functions on $X$ and put $f(x)=\lim_{n\to \infty}f_n(x)$ for $x\in X$. Then $f$ is $D_\mu$-integrable iff the sequence $(\int f_n\,d\mu)_{n\geq 0}$ is bounded in $\R$.

  \end{theorem}

\begin{proof}
Building on Lemma~\ref{saks}, the proof is similar to the Henstock-Kurzweil case.

\end{proof}

Recall that the Lebesgue integral of a non-negative function $f:X\to \R$ is defined as the supremum of the Lebesgue integral of non-negative simple functions below it: \[{\cal L}\int f\,d\mu=\sup\left\{ \int s\,d\mu:\mbox{simple }s\leq f\right\},\]
  where for a simple function $s=\sum_{1\leq i\leq n}a_i\chi_{S_i}$ we have $\int s\,d\mu=\sum_{1\leq i\leq n}a_i\mu(S_i)$. 
  For $f:X\to \R$, let $f_+,f_-:X\to \R$ be given by $f_+(x)=\max\{0,f(x)\}$ and $f_-(x)=\min\{f(x),0\}$. Then $f$ is said to be {\em Lebesgue integrable} with value $r\in\R$ if ${\cal L}\int f_+\,d\mu$ and ${\cal L}\int (-f_-)\,d\mu$ both exist and $r={\cal L}\int f_+\,d\mu-{\cal L}\int (-f_-)\,d\mu$. Since characteristic functions are $D_\mu$-integrable, it follows from Proposition~\ref{basic-int-prop}(ii) that simple functions are $D_\mu$-integrable. From the Monotone Convergence Theorem~\ref{mct}, we obtain:
  \begin{theorem}\label{lebesgue}
    If $f:X\to \R$ is Lebesgue integrable with respect to $\mu$, then it is $D_\mu$-integrable and the values of the two integrals coincide.
  \end{theorem}\begin{proof}
If $f$ is Lebesgue integrable with value $r$ then $f_+,f_-$ are both Lebesgue integrable with $r={\cal L}\int f_+\,d\mu-{\cal L}\int (-f_-)\,d\mu$. Sine the Lebesgue integral of a simple function with respect to $\mu$ coincides with its $D_\mu$-integral, it follows by the Monotone Convergence Theorem~\ref{mct}, that the $D_\mu$-integral of $f_+,f_-$ coincide with their Lebesgue integral respectively. By Proposition~\ref{basic-int-prop}(ii), we obtain $r=\int f_+\,d\mu-\int (-f_-)\,d\mu$, as required.
    \end{proof}
  We also have an extension of the convergence theorems of the Henstock-Kurzweil integration theory for uniform convergence, Fatou's lemma and dominated convergence with similar proofs, which we skip here.

  \section{Basis change in integration on unit interval}
  We now consider the $D_\mu$ integral of unctions on the unit interval, i.e., $\mu$ is a normalised measure on $[0,1]$. A particular case is given by the well-known Henstock-Kurzweil-Stieljes integration. If the normalised Borel measure $\mu$ is absolutely continuous with respect to the Lebesgue measure $\lambda$, i.e. $\mu(A)=$ whenever $\lambda(A)=0$ for any Borel set $A\subset [0,1]$ and if ${\cal B}_1$ is the set of all intervals of $[0,1]$ as in Example~\ref{unit}, then we obtain the Henstock-Kurzweil-Stieljes integration of the form $\int f(t)\,d\phi(t)$ where $\phi(t)=\mu([0,t])$.

 We show here generally that the $D_\mu$ integral of a function on the unit interval with respect to any normalised Borel measure does not change if we change the open basis ${\cal B}$ in defining the integral. We note that any open basis  ${\cal B}$ induces a dense subset on $[0,1]$ and  ${\bf C}{\cal B}$ consists precisely of intervals with endpoints in this dense set.

   Suppose now we take any countable dense set $S\subset [0,1]$ and let ${\cal B}_2$ be the set of all (open, closed or half-open/half-closed) intervals with endpoints in $S$. We now have two notions of $D_\mu$-integrability one, namely the classical definition, with respect to ${\cal B}_1$ and one with respect to ${\cal B}_2$. To fix our ideas, we write ``$\int f\,d\mu$ (${\bf C}_{{\cal B}_i}$)'' to mean the value of the $D_\mu$-integral with respect to the basis ${\cal B}_i$, $i=1,2$. For a non-empty interval $R\subset [0,1]$, let $R^-$ and $R^+$ denote its left and right endpoints respectively. 

\begin{proposition}\label{invariant-HK}
Given any two collections of crescents generated by two bases ${\cal B}_1$ and ${\cal B}_2$, a map $f:[0,1]\to \R$ is  $D_\mu$ integrable with respect to ${\bf C}_{{\cal B}_1}$ iff it is $D_\mu$-integrable with respect to ${\bf C}_{{\cal B}_2}$ and
  \[  \int f\,d\mu\, ({\bf C}_{{\cal B}_1})=  \int f\,d\mu\, ({\bf C}_{{\cal B}_2})      \]
\end{proposition} \begin{proof}
 It is sufficient to prove the result when ${\cal B}_1$ is taken to be the collection of all open sets.  In fact, since ${\cal B}_2\subset {\cal B}_1$, it is clear that if $f$ is $D_\mu$-integrable, then it is $D_\mu$-integrable with respect to ${\cal B}_2$. Suppose therefore that $f$ is $D_\mu$-integrable with respect to ${\cal B}_2$ with integral value $r$ and assume ${\cal B}_2$ induces the dense subset $T\subset[0,1]$ so that any element of ${\bf C}_{{\cal B}_2}$ is precisely an interval with endpoints in $T$. Let $\epsilon>0$ be given. Let $\gamma$ be a gauge such that $\dot{Q}\prec \gamma$ implies:
  \[\left|\int f\,d\mu_{\dot{Q},\gamma} ({\bf C}_{{\cal B}_2}) -r\right|<\epsilon/2\]
  Let $\dot{Q}= \{(C_j,s_j):1\leq j\leq m\}\prec \gamma$ with $C_j\in {\bf C}_{{\cal B}_1}$. We can assume $s_j\neq s_{j+1}$ for $1\leq j\leq m-1$, since otherwise we can merge the two intervals $C_k$ and $C_{k+1}$. We now convert the tagged partition $\dot{Q}$ to a tagged partition $\dot{Q'}$, with the same tags, such that the following conditions hold.
  \begin{itemize}
  \item[(i)] The intervals in $Q'$ are in ${\bf C}_{{\cal B}_2}$, i.e., have endpoints in $T$.
  \item[(ii)] $\dot{Q'}\prec \gamma$.
  \item[(iii)] \begin{equation}\label{change-base}\left|\int f\,d\mu_{\dot{Q'},\gamma}({\bf C}_{{\cal B}_2}) -\int f\,d\mu_{\dot{Q},\gamma}({\bf C}_{{\cal B}_1}) \right|<\epsilon/2\end{equation}
  \end{itemize}
  
  From these three conditions, it will then follow that $f$ is $D_\mu$-integrable with respect to ${\cal B}_1$ and  $ \int f\,d\mu\, ({\bf C}_{{\cal B}_1})=  \int f\,d\mu\, ({\bf C}_{{\cal B}_2})$ as required. Let $K>0$ be such that $|f(s_j)|\leq K$ for $1\leq j\leq m$. Since any normalised Borel measure on a compact metric space is regular, there exists $\delta>0$, with $\delta<\min_{1\leq j\leq m} (C_j^+-C_j^-)/3$, such that $\mu([C_j^{\pm}-\delta, C^{\pm}_j+\delta])-\mu(\{C_j^{\pm}\})<\epsilon/4mK$, which implies that $\mu([C_j^{\pm}-\delta, C^{\pm}_j))<\epsilon/4mK$ and $\mu((C_j^{\pm}, C_j^{\pm}+\delta])<\epsilon/4mK$ for $1\leq j\leq m$.  The procedure to obtain $\dot{Q'}=\{(C_j',s_j): 1\leq j\leq m\} $ from  $\dot{Q}$, which preserves the set of tags $\{s_j:1\leq j\leq m\}$, is applied to each pair of adjacent intervals $C_j$ and $C_{j+1}$ (for $1\leq j\leq m-1$) in $Q$ by slightly expanding or shrinking these two intervals $C_j$ and $C_{j+1}$ to obtain $C_j'$ and $C_{j+1}'$. This is achieved by slightly moving the boundary point $p_j:=C_j^+=C_{j+1}^-$, either left or right, to the new boundary point $p_j'$ with $p_j\in C_j$ iff $p_j'\in C'_j$.

  The detailed procedure for the adjacent pair $C_j$ and $C_{j+1}$ with the boundary point $p_j$ is as follows.
    Let $\delta_0:=\min\{\delta, \gamma(p_j):{1\leq j\leq m-1}\}$.
    \begin{itemize} \item[(i)] If $p_j\in C_j$, then replace the boundary point $p_j$ with a point 
    $p'\in T\cap (p_j,p_j+\delta_0)$ with $p_j'<s_{k+1}$ and $p'_j-s_j<\gamma(s_j)$.
    \item[(ii)] If $p_j\in C_{j+1}$, then replace the boundary point $p_j$ with a point 
    $p'\in T\cap (p_j-\delta_0,p_j)$ with $p_j'>s_{k+1}$ and $p_j-p'_j<\gamma(s_{j+1})$.
 \end{itemize}
By construction, we have $\dot{Q}\prec \gamma$ and the estimate: $|\mu(C_j)-\mu(C_j') |\leq \epsilon/2mK$ for $1\leq j\leq m-1$. Thus:
    \[\left|\int f\,d\mu_{\dot{Q},\gamma} ({\bf C}_{{\cal B}_1}) -\int f\,d\mu_{\dot{Q'},\gamma} ({\bf C}_{{\cal B}_2})\right|\]\[\leq \left| \sum_{j=1}^{m-1}f(s_j)(\mu(C_j)-\mu(C_j'))    \right|\leq  \sum_{j=1}^{m-1}|f(s_j)||\mu(C_j)-\mu(C_j') |   \leq \sum_{j=1}^{m-1} K\epsilon/2mK\leq\epsilon/2. \]
    We conclude that
    
    \[\left|\int f\,d\mu_{\dot{Q},\gamma} ({\bf C}_{{\cal B}_1}) -r\right|\leq \left|\int f\,d\mu_{\dot{Q},\gamma} ({\bf C}_{{\cal B}_1}) -\int f\,d\mu_{\dot{Q'},\gamma} ({\bf C}_{{\cal B}_2})\right|+\left| \int f\,d\mu_{\dot{Q'},\gamma} ({\bf C}_{{\cal B}_2})-r\right| 
    <\epsilon/2+\epsilon/2=\epsilon\]

\end{proof}
The above result can be extended to $[0,1]^n$; we will skip the details.

\section{$D_\mu$-integrable non-Lebesgue integrable maps}\label{d-int-non-leb}
We will construct families of $D_\mu$-integrable but non-Lebesgue integrable functions on the Cantor space $\{0,1\}^\omega$ of Example~\ref{cantor}.

Consider the normalised Borel measure given by the unique invariant measure of the contracting Iterated Function System with probabilities~\cite{Hut81} given by the two contracting maps $h_0,h_1:\{0,1\}^\omega\to \{0,1\}^\omega$ defined by 
\[f_0(x)=0x\qquad f_1(x)=1x\]
and probabilities $p_0,p_1\geq 0$ with $p_0+p_1=1$. The invariant measure is the unique solution of the recursive equation
\[\mu=p_0{\mu\circ f_0^{-1}}+p_1{\mu\circ f_1^{-1}}\]
in ${\bf M}^1\{0,1\}^\omega$. The map $g:\{0,1\}^\omega\to [0,1]$ with
\[g(x)=\sum_{n=0}^\infty\frac{x_n}{2^{n+1}}\]
is a continuous surjective map which is one-to-one except that it sends any pair of elements of the form $y0\overline{1}$ and $y1\overline{0}$, for any finite sequence $y=y_0\ldots y_{k-1}\in \{0,1\}^k$ with some positive integer $k\geq 1$, to the same point. The set
$A:=\{y0\overline{1},y1\overline{0}:y\in \{0,1\}^k,k\in{\mathbb N}\}$
  is countable and has $\mu$-measure zero. The two maps $u_i:[0,1]\to [0,1]$ with $u_i(r)=\frac{r+i}{2}$ for $i=0,1$ satisfy: $g\circ h_i=r_i\circ g$ for $i=0,1$. 
The map $g$ sends the cylinder clopen set $B_{a_0\ldots a_{n-1}}$, which is a ball of diameter $1/2^n$ in $\{0,1\}^\omega$, to the dyadic interval \[g[B_{a_0\ldots a_{n-1}}]=\bigcap_{i=0}^{n-1}r_{a_i}\circ r_{a_{i-1}}\circ \ldots r_{a_0}[0,1],\]
which also has diameter $1/2^n$. This establishes a one-to-one relation between cylinder clopen sets and dyadic intervals. 

The normalised Borel measure $\mu\in \{0,1\}^\omega$ induces a Borel measure on $[0,1]$ via $g$ given by $\mu\circ g^{-1}\in {\bf M}^1[0,1]$. Assume $p_0=p_1=1/2$ from now on. (For $p_0\neq 1/2$, the measure $\mu\circ g^{-1}$ is singular with respect to the Lebesgue measure on $[0,1]$.) Thus, for any dyadic interval $I=\bigcap_{i=0}^{n-1}r_{a_i}\circ r_{a_{i-1}}\circ \ldots r_{a_0}[0,1]$ of length $1/2^n$ we have $\mu\circ g^{-1}(I)=1/2^n$ and, hence, $\mu\circ g^{-1}=\lambda$, the Lebesgue measure.

Let $f:[0,1]\to \R$ be any function. Recall that the  $D_\mu$-integration using the collection of all intervals of $[0,1]$ as the partitionable set of crescents (Example~\ref{unit}) is equivalent to Henstock-Kurzweil integration. 
\begin{proposition}\label{cantor-hk}
The map $f:[0,1]\to \R$ is Henstock-Kurzweil integrable if and only if $f\circ g:\{0,1\}^\omega\to \R$ is  $D_\mu$--integrable.

  \end{proposition}

\begin{proof} By Lemma~\ref{invariant-HK}, we can take the collection of all dyadic intervals of $[0,1]$ as the partitionable set ${\bf C}_{\cal B}$ of crescents. Suppose first that $f$ is Henstock-Kurzweil integrable and let $\epsilon>0$ be given. There exists a gauge $\gamma:[0,1]\to \R$ such that whenever $\dot{P}\prec \gamma$, we obtain: \begin{equation}\label{hk}\left|\int f\,d\lambda_{\dot{P},\gamma}-\int f\,d\lambda\right|<\epsilon.\end{equation}
  Define a gauge $\gamma':\{0,1\}^\omega\to \R$ as follows. Let $x\in \{0,1\}^\omega$. By the continuity of $g$ there exists $\delta_x>0$ such that $g[O_\delta(x)]\subset O_{g(x)}(\gamma(g(x)))$. Put $\gamma'(x)=\delta_x$. If $\dot{Q}\prec \gamma'$ with $\dot{Q}=\{(R_i,t_i):1\leq i\leq n\}$, then we claim that $\dot{P}:=\{(g[R_i],g(t_i)),1\leq i\leq n\}$ is a $\gamma$-fine tagged partition of $[0,1]$. In fact, since each $R_i$ is a cylinder clopen sets, the closed intervals $g[R_i]$, for $1\leq i\leq n$ are dyadic intervals with disjoint interiors. By surjectivity of $g$ we have $\bigcup_{1\leq i\leq n}g[R_i]=[0,1]$. (If $g(t_i)=g(t_j)$ for $i\neq j$, then, since $R_i\cap R_j=\emptyset$, we have $t_i\neq t_j$, i.e., $t_i,t_j\in A$). By construction of $\gamma'$, we have  $\dot{P}\prec \gamma$. Thus, $\dot{P}$ satisfies Equation~\ref{hk}. On the other hand, we have:
  
  \[\int f\circ g\,d\mu_{\dot{Q},\gamma} =\sum_{i=1}^nf\circ g(t_i)\mu(R_i)=\sum_{i=1}^nf(g(t_i))\mu(g[R_i])\]\[=\int f\,d\lambda_{\dot{P},\gamma},\]
  since $\mu(R_i)=\mu(g^{-1}(g[R_i])$ as $R_i\subset g^{-1}(g[R_i])$ and we have: $g^{-1}(g[R_i])\setminus R_i\subset A$ has $\mu$ measure zero. Thus, $f\circ g$ is $D_\mu$-integrable. 
  
  Now suppose $f\circ g$ is $D_\mu$-integrable and let $\epsilon>0$ be given. There exists a gauge $\gamma:\{0,1,\}^\omega$ such that $\dot{P}\prec \gamma$ implies \[\left|\int f\circ \,d\mu_{\dot{P},\gamma}-\int f\circ g\,d\mu \right|<\epsilon.\]
  Define a gauge $\gamma':[0,1]\to\R$. If $y\notin g[A]$ (i.e., $y$ is not a dyadic number), then there exists a unique $x\in \{0,1\}^\omega$ with $y=g(x)$. Let $B_{a_0\ldots a_{n-1}}\subset O_{\gamma(x)}(x)$  for some sequence $a_0\ldots a_{n-1}\in \{0,1\}^n$ and some $n\in{\mathbb N}$. Then $y\in I^\circ$ where $I:=g[B_{a_0\ldots a_{n-1}}]$. Put $\gamma'(y)=\min\{|I^--x|,|I^+-x|\}$. If, on the other hand, $y\in g[A]$, then there exits $x_1,x_2\in A$ with $y=g(x_1)=g(x_2)$. Let $m\in{\mathbb N}$ be such that $1/2^m<\min\{\gamma(x_1),\gamma(x_2)\}$ and put $\gamma'(y)=1/2^m$. If $\dot{Q}=(R_i,t_i):1\leq i\leq k\}\prec \gamma'$ is a dyadic partition of $[0,1]$, then there exists a unique $\gamma$-fine tagged partition $\dot{P}=\{(S_i,s_i): 1\leq i\leq k\}$ such that $g[S_i]=R_i$ and $g(s_i)=t_i$ for $1\leq i\leq k$. Thus,
  \[\left |\int f\,d\lambda_{\dot{Q},\gamma'}-\int f\circ g\,d\mu\right|= \left|\int f\circ g\,d\lambda_{\dot{P},\gamma}-\int f\circ g\,d\mu\right|<\epsilon\]

\end{proof}

\begin{example}
  Let $f:[0,1]\to \R$ be given by $f(x)=\frac{1}{x}\sin \frac{1}{x^3}$. Then $g$ is Henstock-Kurzweil integrable and in fact it has an improper Riemann integral. By the change of variable $u=1/x^3$, one can show that the following limit of the Riemann integral exists:

  \[\lim_{a\to 0}\int_a^1 f(x)\,dx= \frac{\pi}{6}-\frac{1}{3} \int_0^1\frac{\sin u}{u}\,du,\]

which is the value of the Henstock-Kurzweil integral of $f$ on $[0,1]$. However, $f$ is not absolutely integrable and is thus not Lebesgue integrable. Therefore, the map $f\circ g:\{0,1\}^\omega\to \R$ is $D_\mu$-integrable but not Lebesgue integrable. 

  \end{example}

\bibliographystyle{plain}
\bibliography{biblio,CCCRandomVariables}
\newpage

\section*{Appendix}
We provide here the basic domain-theoretic notions required in this paper. 
A {\em directed complete partial order} (dcpo) $D$ is a partial order in which every (non-empty) directed set $A\subseteq D$ has a lub (least upper bound) or supremum $\sup A$. The {\em way-below relation} $\ll$ in a dcpo $(D,\sqsubseteq)$ is defined by $x\ll y$ if whenever there is a directed subset $A\subseteq D$ with $y\sqsubseteq \sup A$, then there exists $a\in A$ with $x\sqsubseteq a$.  A subset $B\subseteq D$ is a {\em basis} if for all $y\in D$ the set $\{x\in B:x\ll y\}$ is directed with lub $y$. By a {\em domain}, we mean a non-empty dcpo with a countable basis. Domains are also called {\em $\omega$-continuous dcpo's}. The Scott topology on a domain $D$ with basis $B$ has sub-basic open sets of the form $\dua b:=\{x\in D:b\ll x\}$ for any $b\in B$. 

The set of non-empty compact intervals of the real line ordered by reverse inclusion and augmented with the whole real line as bottom is the prototype bounded complete domain for real numbers denoted by $\IR$, in which $I\ll J$ iff $J \subseteq I^\circ$. It has a basis consisting of all intervals with rational endpoints.  For two non-empty compact intervals $I$ and $J$, their infimum $I\sqcap J$ is the convex closure of $I\cup J$. A basis of the Scott topology of $\IR$ is given by $\dua J=\{I\in \IR:I\subset J^\circ\}$ for any compact interval $J$. 

If $X$ is a second countable locally compact Hausdorff space then the {\em upper space} of $X$ is defined as the set of non-empty compact subsets of $X$ ordered by reverse inclusion. It is a domain, in which the supremum of a directed set $(A_i)_{i\in I}$ is the intersection $\bigcap_{i\in I}A_i$, which is non-empty and compact. The way-below relation is given by $A\ll B$ iff $A^\circ \supset B$, where $A^\circ$ is the interiors of $A$.

Let $\mathcal {O}(Y)$ be the lattice of open subsets of a topological space $Y$. Recall
from~\cite{birkhoff1940lattice,Saheb80,lawson1982valuations,JonesP89} that a {\em valuation} on a topological
space $Y$ is a map
$\nu:\mathcal{O}(Y)\to [0,1]$ which satisfies:
\begin{itemize}
\item[(i)] $\nu(a)+\nu(b)=\nu(a\cup b)+\nu(a\cap b)$
\item[(ii)] $\nu(\emptyset)=0$
\item[(iii)] $a\subseteq b\to \nu(a)\leq \nu(b)$
\end{itemize}

A {\em continuous} valuation~\cite{lawson1982valuations,JonesP89} is a valuation such that whenever
$A\subseteq\mathcal{O}(Y)$ is a directed set (wrt $\subseteq$) of open sets
of $Y$, then
\[\nu\left(\bigcup_{O\in A}O\right)=\mbox{sup}\,_{O\in A}\nu(O).\] We will work with normalised continuous valuations of a domain $D$, i.e., those with unit mass on the whole space $D$.

For any $b\in Y$, the {\em point valuation} based at $b$ is the
valuation $\delta_b:\mathcal{O}(Y)\to [0,1]$ defined by
\[\delta_b(O)=\left\{\begin{array}{cl}
1&\mbox{if $b\in  O$}\\
0&\mbox{otherwise}.\end{array}\right.\]
Any finite linear combination
$\sum_{i=1}^{n}r_i\delta_{b_i}$
of point valuations $\delta_{b_i}$ with constant coefficients
$r_i\in [0,\infty)$, ($1\leq i\leq n$) is a continuous valuation on
$Y$, called a {\em simple valuation}. 

The {\em probabilistic power domain}, ${\bf P}^1Y$, of a topological space $Y$ consists of
the set of normalised continuous valuations $\nu$ on $Y$ with $\nu(Y)= 1$
and is ordered as follows:
\[\mu\sle\nu\mbox{ iff for all open sets $O$ of $Y$, }\mu(O)\leq \nu(O).\]
The partial order $({\bf P}^1Y,\sle)$ is a dcpo with bottom in which  the lub of a directed set
$\langle\mu_i\rangle_{i\in I}$ is given by $\sup_i\mu_i=\nu$, where
for $O\in \mathcal{O}({Y})$:
\[\nu(O)=\mbox{sup}\,_{i\in I}\mu_i(O).\]

 These will correspond to probability distributions on $D$~\cite{alvarez2000extension}. Consider the {\em normalised} probabilistic power domain ${\bf P}^1D$ of a domain $D$, consisting of normalised continuous valuations with pointwise order. ${\bf P}^1D$ is an $\omega$-continuous dcpo, with a countable basis consisting of normalized valuations given by a finite sum of pointwise valuations with rational coefficient~\cite{Eda94}.

 From the above, it follows that for any compact metric space $X$, the probabilistic power domain ${\bf P}^1{\bf U}X$  of the upper space ${\bf U}X$ of $X$ is a domain. 

 The splitting lemma for normalised valuations~\cite{Eda94} with respect to the information ordering $\sle$, which is similar to the splitting lemma for valuations~\cite{JonesP89}, states: If \[\alpha=\sum_{1\leq i\leq m}p_i\delta(c_i),\quad\beta=\sum_{1\leq j\leq n}q_j\delta(d_j),\]are two normalised valuations on a continuous dcpo $D$ then $\alpha\sqsubseteq\beta$ iff there exist $t_{ij}\in [0,1]$ for $1\leq i\leq m$, $1\leq j\leq n$ such that
\begin{itemize}
\item $\sum_{i=1}^m t_{ij}=q_j$ for each $j=1, \ldots, n$.
\item $\sum_{j=1}^n t_{ij}=p_i$ for each $i=1, \ldots, m$.
\item $t_{ij}>0\,\Rightarrow\,c_i\sqsubseteq d_j$. 
\end{itemize}

 The splitting lemma for normalised valuations~\cite{Eda94} with respect to the way-below relation $\ll$, which is similar to the splitting lemma for valuations~\cite{JonesP89}, is as follows.

 Suppose $\alpha=\sum_{1\leq i\leq m}p_i\delta(c_i)$ and $\beta=\sum_{1\leq j\leq n}q_j\delta(d_j)$ are normalised valuations on a continuous dcpo. 
\begin{proposition}\label{sp-way}\cite[Proposition 3.5]{Eda94} We have $\alpha\ll\beta$ if and only if there exist $t_{ij}\in [0,1]$ for $1\leq i\leq m$ and $1\leq j\leq n$ such that 
\begin{itemize}
\item $c_{i_0}=\bot$ for some $i_0$ with $1\leq i_0\leq  m$ and for all $j$ with $1\leq j\leq m$, we have $t_{i_0j}>0$,
\item $\sum_{i=1}^m t_{ij}=q_j$ for each $j=1, \ldots, n$,
\item $\sum_{j=1}^n t_{ij}=p_i$ for each $i=1, \ldots, m$.\
\item $t_{ij}>0\,\Rightarrow c_i\ll d_j$. 
\end{itemize}
\end{proposition}

\end{document}